\documentclass[a4paper,12pt]{amsart}

\input xypic
\usepackage{amssymb}
\xyoption{all}

\newtheorem{thm}{Theorem}[section]
\newtheorem{lemma}[thm]{Lemma}
\newtheorem{cor}[thm]{Corollary}
\newtheorem{prop}[thm]{Proposition}
\theoremstyle{definition}
\newtheorem{rem}[thm]{Remark}
\newtheorem{hyp}{Hypothesis}
\def\Mbar{\bar{M}}
\def\Nbar{\bar{N}}
\def\fbar{\bar{f}}
\def\GA{{\mathfrak{A}}}
\def\CA{{\mathcal{A}}}
\def\CC{{\mathcal{C}}}
\def\CF{{\mathcal{F}}}
\def\CG{{\mathcal{G}}}
\def\CS{{\mathcal{S}}}
\def\CT{{\mathcal{T}}}

\def\BZ{{\mathbf{Z}}}
\def\tf{{\tilde{f}}}

\def\add{\operatorname{add}\nolimits}
\def\Aut{\operatorname{Aut}\nolimits}
\def\can{{\mathrm{can}}}
\def\coker{\operatorname{coker}\nolimits}
\def\cone{\operatorname{cone}\nolimits}
\def\End{\operatorname{End}\nolimits}
\def\Ext{\operatorname{Ext}\nolimits}
\def\gr{{\operatorname{gr}\nolimits}}
\def\Hom{\operatorname{Hom}\nolimits}
\def\id{\operatorname{id}\nolimits}
\def\im{\operatorname{im}\nolimits}
\def\mMod{\operatorname{\!-mod}\nolimits}
\def\mstab{\operatorname{\!-stab}\nolimits}
\def\opp{{\operatorname{opp}\nolimits}}

\def\rad{\operatorname{rad}\nolimits}
\def\Res{\operatorname{Res}\nolimits}
\def\eps{\varepsilon}
\def\iso{\buildrel \sim\over\to}
\def\Rarr#1{\buildrel #1\over \longrightarrow}
\def\ie{{\em i.e.}}

\author{Jeremy Rickard}
\address{Jeremy Rickard~: University of Bristol, School of Mathematics, University Walk,
Bristol BS8 1TW, UK}
\email{J.Rickard@bristol.ac.uk}
\author{Rapha\"el Rouquier}
\address{Rapha\"el Rouquier~: Mathematical Institute,
University of Oxford, 24-29 St Giles', Oxford, OX1 3LB, UK}
\email{rouquier@maths.ox.ac.uk}

\title{Stable categories and reconstruction}
\begin{document}
\maketitle
\section{Introduction}
This work is an attempt towards a Morita theory for stable equivalences
between self-injective algebras. More precisely, given two
self-injective algebras $A$ and $B$ and an equivalence between their
stable categories, consider the set $\CS$ of images of simple 
$B$-modules inside the stable category of $A$. That set satisfies
some obvious properties of $\Hom$-spaces and it generates the stable category
of $A$. Keep now only $\CS$ and $A$. Can $B$ be reconstructed ? We show
how to reconstruct the graded algebra associated to the radical
filtration of (an algebra Morita equivalent to) $B$.
It would be interesting to develop further an obstruction theory
for the existence of
an algebra $B$ with that given filtration, starting only with $\CS$ 
(this might be studied in terms of localization of $A_\infty$-algebras).
Note that a result of Linckelmann \cite{Li} shows that, if we consider only
stable equivalence of Morita type, then $B$ is characterized by $\CS$
--- but this result does not provide a reconstruction of $B$ from $\CS$.

We also study a similar problem in the more general setting of a triangulated
category $\CT$. Given a finite set $\CS$
of objects satisfying $\Hom$-properties
analogous to those satisfied by the set of simple modules in the derived
category of a ring and assuming that the set generates $\CT$,
we construct a $t$-structure on $\CT$. In the case $\CT=D^b(A)$ and
$A$ is a symmetric algebra, the first author has shown \cite{Ri} that
there is a symmetric algebra $B$ with an equivalence
$D^b(B)\iso D^b(A)$ sending the set of simple $B$-modules to $\CS$.
The case of a self-injective algebra leads to a slightly more general
situation~: there is a finite dimensional differential graded algebra $B$ with
$H^i(B)=0$ for $i>0$ and for $i\ll 0$ with the same property as above.

\section{Notations}
Let $\CC$ be an additive category. Given $S$ a set of objects of $\CC$,
we denote by $\add S$ the full subcategory of $\CC$ of objects isomorphic
to finite direct sums of objects of $S$.

Let $k$ be a field and $A$ a finite dimensional $k$-algebra. We say
that $A$ is split if the endomorphism ring of every simple $A$-module is
$k$.
We denote by $A\mMod$ the category of finitely generated
left $A$-modules and by $D^b(A)$ its derived category. For $A$ self-injective,
we denote by $A\mstab$ the stable category, the quotient of $A\mMod$ by projective
modules. Given $M$ an $A$-module, we denote by $\Omega M$
the kernel of a projective cover of $M$ and by $\Omega^{-1}M$ the
cokernel of an injective hull of $M$.

\section{Simple generators for triangulated categories}
\subsection{Category of filtered objects}
\label{filtT}
Let $\CT$ be a triangulated category and $\CS$ a full subcategory of $\CT$.

We define a category $\CF$ as follows.

$\bullet\ $Its objects are
diagrams
$$M=(\cdots\to M_2\Rarr{f_2} M_1\Rarr{f_1} M_0\Rarr{\eps_0}N_0)$$
where $M_i$ is an object of $\CT$, $M_i=0$ for $i\gg 0$, such that
\begin{itemize}
\item[(i)]
$M_1\Rarr{f_1} M_0\Rarr{\eps_0}N_0$ is the beginning of a distinguished
triangle
\item[(ii)]
for all $i\ge 1$, the cone $N_{i-1}$ of $f_i$ is in $\add\CS$
\item[(iii)]
the canonical map $\Hom(N_0,S)\to \Hom(M_0,S)$ is surjective for all
$S\in\CS$
\item[(iv)]
the canonical map $\Hom(N_i,S)\to \Hom(M_i,S)$ is bijective for all
$S\in\CS$ and $i\ge 1$.
\end{itemize}

Note that $\eps_i:M_i\to N_i=\cone(f_{i+1})$
is well defined up to unique isomorphism for $i\ge 1$ thanks to property (iv).
For $i\ge 0$, we define a new object $M_{\ge i}$ of $\CF$ as
$\cdots \to M_{i+1}\Rarr{f_{i+1}} M_i\Rarr{\eps_i} N_i$.

\smallskip
$\bullet\ $Given another
diagram $M'$,
we define $\Hom_\CF(M,M')_0$ as the subspace of
$\Hom(N_0,N'_0)$ consisting of those maps $g$ such that there is
$h:M_0\to M'_0$ with $\eps'_0h=g\eps_0$.

We put $\Hom_\CF(M,M')_i=\Hom_\CF(M,M'_{\ge i})_0$ and
$\Hom_\CF(M,M')=\oplus_{i\ge 0}\Hom_\CF(M,M')_i$.

\smallskip
$\bullet\ $Let now $g_0\in\Hom(N_0,N'_0)$.
By (iv), there are maps $h_0,h_1,\ldots$ and $g_1,g_2,\ldots$
making the following diagrams commutative
$$\xymatrix{
M_{i+1}\ar[r]^{f_{i+1}}\ar[d]_{h_{i+1}} & M_i\ar[d]_{h_i} \ar[r]^{\eps_i} &
 N_i\ar[r]\ar[d]_{g_i} & M_{i+1}[1]\ar[d]_{h_{i+1}[1]}\\
M'_{i+1}\ar[r]^{f'_{i+1}} & M'_i \ar[r]^{\eps'_i} & N'_i\ar[r] & M'_{i+1}[1]
}$$

\begin{lemma}
The maps $g_i:N_i\to N'_i$ (for $i\ge 1$) depend
only on $g_0$.
\end{lemma}

\begin{proof}
Let us show that $g_1$ depends only on $g_0$. The general case is similar,
by induction.

Let $p:M_0\to M'_1$. We have to show that the only map $q:N_1\to N'_1$
such that $\eps'_1pf_1=q\eps_1$ is $q=0$.

By (iii), there is $r:N_0\to N'_1$ such that
$\eps'_1 p=r\eps_0$. So, $\eps'_1pf_1=r\eps_0f_1=0$.
By (iv), we deduce from $q\eps_1=0$ that $q=0$.
$$\xymatrix{
M_1\ar[rrr]^{\eps_1} \ar[rd]^{f_1} &&& N_1\ar@{.>}[dd]^q \\
& M_0\ar[r]^{\eps_0}\ar[dl]_{p} & N_0 \ar[rd]^{r} \\
M'_1\ar[rrr]_{\eps'_1} &&& N'_1 
}$$
\end{proof}

\smallskip
Let $g_0\in\Hom_\CF(M,M')_i$ and $g'_0\in\Hom_\CF(M',M'')_j$.
We define the product
$g'_0g_0$ as the composition
$N_0\Rarr{g_0} N'_i\Rarr{g'_i}N''_{i+j}$.

\begin{lemma}
Assume $\Hom(S,T[n])=0$ for all $S,T\in\CS$ and $n<0$.
Let $M$ be an object of $\CF$. Then, 
the canonical map
$\Hom(N_0,S)\to \Hom(M_0,S)$ is an isomorphism.
\end{lemma}

\begin{proof}
By induction on $-i$, we see that $\Hom(M_i,S[n])=0$ for $n<0$ and
$S\in\CS$. It follows that $\Hom(M_1[1],S)=0$, hence
the canonical map
$\Hom(N_0,S)\to \Hom(M_0,S)$ is injective, as well as being surjective
by assumption.
\end{proof}

\subsection{$t$-structures}
Let $k$ be a field and assume $\CT$ is a $k$-linear triangulated category.

We assume from now on the following
\begin{hyp}
\label{hyp2}
\begin{itemize}
\item[(1)]
$\Hom(S,T)=k^{\delta_{S,T}}$ for $S,T\in \CS$
\item[(2)]
$\CS$ generates $\CT$ as a triangulated category
\item[(3)] $\Hom(S,T[n])=0$ for
$S,T\in\CS$ and $n<0$.
\end{itemize}
\end{hyp}

\subsubsection{}
\label{tstruc}

\begin{lemma}
\label{reorder}
Given $N\in\CT$, there is a sequence
$0=M_r\Rarr{f_r}\cdots\to M_2\Rarr{f_2} M_1\Rarr{f_1} M_0=N$
and $d:\BZ_{>0}\to \BZ$ non increasing
such that $\cone(f_i)[d(i)]\in\CS$.

For such a sequence, the maps
$M_{r-1}\to N$ and $N\to \cone(f_1)$ are non zero.
\end{lemma}

\begin{proof}
Since $\CT$ is generated by $\CS$, there is a sequence
$0=M_r\to\cdots\to M_2\Rarr{f_2} M_1\Rarr{f_1} M_0=N$
and $d:\BZ_{>0}\to \BZ$
such that $\cone(f_i)[d(i)]\in\CS$.

We put $N_i=\cone(f_i)=S_i[-d(i)]$ with
$S_i\in\CS$.
Take $i$ such that $d(i)>d(i-1)$.
Let $T$ be the cone of $f_{i-1}f_i:M_i\to M_{i-2}$.
The octahedral axiom gives a distinguished triangle
$S_i[-d(i)]\to T\to S_{i-1}[-d(i-1)]\rightsquigarrow$. 

Assume
the morphism $S_{i-1}[-d(i-1)]\to S_i[-d(i)+1]$ is
non zero. Then it is an isomorphism and  $d(i)=d(i-1)+1$.
It follows that $T=0$ and $f_{i-1}f_i$ is an isomorphism.
Consequently,
$$0=M_r\to\cdots\to M_{i+1}\Rarr{f_{i-1}f_if_{i+1}}M_{i-2}
\to\cdots\to M_2\Rarr{f_2} M_1\Rarr{f_1} M_0=N$$
is a new sequence with successive cones being shifts
of objects of $\CS$.

By induction, we can assume that the
morphism $S_{i-1}[-d(i-1)]\to S_i[-d(i)+1]$
is zero. Then, $T\simeq N_i\oplus N_{i-1}$. There
is an object $M'_{i-1}$ and distinguished triangles
$M_i\to M'_{i-1}\to N_{i-1}\rightsquigarrow$ and
$M'_{i-1}\to M_{i-2}\to N_i\rightsquigarrow$.
Put $M'_j=M_j$ for $j\not=i-1$.
So,
$$0=M'_r\to\cdots \to M'_2\to M'_1\to M'_0=N$$
is a new sequence with the same cones
as in the original sequence except the $i$ and $i-1$ ones
which have been swapped.
By induction, we can reorder the cones in the sequence
so that $d$ is non increasing.

\smallskip
Assume the map $M_{r-1}\to N$ is zero. Let 
$T$ be its cone. Then $T\simeq N\oplus M_{r-1}[1]$.
Note that $T$ is filtered by the $S_i[-d(i)]$ with
$-d(i)< -d(r)+1$, hence
$\Hom(M_{r-1}[1],T)=0$. So we have a contradiction.
The case of the map $N\to N_1$ is similar.
\end{proof}

Let $\CT^{\le 0}$ (resp. $\CT^{>0}$)
be the full subcategory of objects $N$ in $\CT$ such that there is
a sequence $0=M_r\to\cdots\to M_2\Rarr{f_2}M_1\Rarr{f_1}M_0=N$ with
$\cone(f_i)$ a direct sum of objects $S[r]$ with $S\in\CS$ and $r\ge 0$
(resp. $r<0$).

\begin{prop}
$(\CT^{\le 0},\CT^{>0})$ is a bounded t-structure on $\CT$.
\end{prop}

\begin{proof}
By induction, we see there is no non-zero map
from an object of $\CT^{\le 0}$ to an object of
$\CT^{>0}$. Furthermore, we have
$\CT^{\le 0}[1]\subseteq \CT^{\le 0}$ and
$\CT^{>0}\subseteq \CT^{>0}[1]$.

Let $N\in\CT$. Pick a sequence as in Lemma \ref{reorder}.
Take $s$ such that $d(s)>0$ and $d(s+1)\le 0$.
Let $L$ be the cone of $f_1\cdots f_s:M_s\to N$.
We have a distinguished triangle
$$M_s\to N\to L\rightsquigarrow$$
with $M_s\in\CT^{\le 0}$ and $L\in\CT^{>0}$.
\end{proof}

We have a characterization of $\CT^{\ge 0}$ and $\CT^{\le 0}$~:
\begin{prop}
\label{charact}
Let $N\in\CT$.
Then, 
$N\in\CT^{\le 0}$ if and only if
$\Hom(N,S[i])=0$ for $S\in\CS$ and $i<0$.

Similarly,
$N\in\CT^{\ge 0}$ if and only if
$\Hom(S[i],N)=0$ for $S\in\CS$ and $i>0$.
\end{prop}

\begin{proof}
We have $\Hom(N,S[i])=0$ for $S\in\CS$ and $i<0$,
if $N\in\CS[r]$ with $r\ge 0$. By induction, it follows
that if $N\in\CT^{\le 0}$, then
$\Hom(N,S[i])=0$ for $S\in\CS$ and $i<0$.

Assume now $\Hom(N,S[i])=0$ for $S\in\CS$ and $i<0$.
Pick a filtration of $N$ as in Lemma \ref{reorder}.
Then, $d(1)\le 0$, hence $d(i)\le 0$ for all $i$
and $N\in\CT^{\le 0}$.

The other case is similar.
\end{proof}

Note that the heart $\CA$ of the $t$-structure is artinian and noetherian.
Its set of simple objects is $\CS$.

\begin{rem}
\label{finiteS}
Assume $\CT$ can be generated by a finite set of objects. Then, 
there is a finite subcategory $\CS'$ of $\CS$ generating $\CT$. It follows
immediately from condition (i) that $\CS=\CS'$.
So, $\CS$ has only finitely many objects.
\end{rem}

\subsubsection{}
\label{fdim}
In \S \ref{fdim}, we
assume $\CT=D^b(A)$ where $A$ is a finite dimensional $k$-algebra.
By Remark \ref{finiteS}, $\CS$ is finite
(note that $\CT$ is generated by the simple $A$-modules, up to isomorphism).

\begin{prop}
\label{construction}
Let $S\in\CS$.
There is a bounded complex of finitely generated injective $A$-modules
$I_\CS(S)\in \CT^{\ge 0}$ such that, given $T\in\CS$ and $i\in\BZ$, we have
$$\Hom_{D^b(A)}(T,I_\CS(S)[i])=
\begin{cases}
 k&\textrm{ for }i=0 \textrm{ and }S=T\\
0& \textrm{ otherwise.}
\end{cases}$$
Similarly,
there is a bounded complex of finitely generated projective $A$-modules
$P_\CS(S)\in \CT^{\le 0}$ such that, given $T\in\CS$ and $i\in\BZ$, we have
$$\Hom_{D^b(A)}(P_\CS(S)[i],T)=
\begin{cases}
 k&\textrm{ for }i=0 \textrm{ and }S=T\\
0& \textrm{ otherwise.}
\end{cases}$$
\end{prop}

\begin{proof}
The construction of a complex $I_{\CS}(S)$ of 
$A$-modules with the $\Hom$ property is \cite[\S 5]{Ri} (note that
the proof of \cite[Lemma 5.4]{Ri} is valid for non-symmetric
algebras).  It is in $\CT^{\ge 0}$ by Proposition \ref{charact}. Since
$\bigoplus_{i\in\BZ} \dim\Hom_{D^b(A)}(V,I_{\CS}(S)[i])=0$ for all simple
$A$-modules $V$, we deduce that $I_{\CS}(S)$ is isomorphic to a bounded
complex of finitely generated injective $A$-modules.

The second case follows from the first one by passing
to $A^\opp$ and taking the $k$-duals of elements of
$\CS$.
\end{proof}

We denote by $\tau^{>0}$, etc... the truncation functors and 
${^tH}^0$ the $H^0$-functor associated to the $t$-structure constructed
in \S \ref{tstruc}.

\begin{lemma}
\label{projinj}
The object ${^tH}^0(I_\CS(S))$ of $\CA$ is an injective hull of
$S$ and ${^tH}^0(P_\CS(S))$ is a projective cover of $S$.
\end{lemma}

\begin{proof}
We have a distinguished triangle
$${^tH}^0(I_\CS(S))\to I_\CS(S)\to \tau^{>0}I_\CS(S) \rightsquigarrow.$$
Let $N\in\CA$. We have 
$\Hom(N,\tau^{>0}I_\CS(S))=0$
and $\Hom(N,I_\CS(S)[1])=0$, so we deduce that
$\Hom(N,{^tH}^0(I_\CS(S))[1])=0$. It follows that
$\Ext^1_\CA(N,{^tH}^0(I_\CS(S)))=0$, hence
${^tH}^0(I_\CS(S))$ is injective.
Since $\Hom(T,(\tau^{>0}I_\CS(S))[-1])=0$, we have
$\Hom(T,{^tH}^0(I_\CS(S)))\iso \Hom(T,I_\CS(S))=k^{\delta_{ST}}$ for $T\in\CS$.
So ${^tH}^0(I_\CS(S))$ is an injective hull of $S$.
The projective case is similar.
\end{proof}

Let us consider the finite dimensional differential graded algebra
$$B=\End^\bullet_A(\bigoplus_S P_\CS(S))=\bigoplus_i
\Hom_A(\bigoplus_S P_\CS(S),\bigoplus_S P_\CS(S)[i]).$$
Denote by $D^b(B)$ the derived category of finite dimensional
differential graded $B$-modules.

\begin{thm}
We have $H^i(B)=0$ for $i>0$ and for $i\ll 0$. We have
$H^0(B)\mMod\simeq \CA$ and
$D^b(B)\simeq D^b(A)$.
\end{thm}

\begin{proof}
Let $N\in\CT$ and consider a filtration of $N$
as in Lemma \ref{reorder}. Take $S\in\CS$ such that
$S[i]$ is isomorphic to the cone of $M_d\to M_{d-1}$.
Then, $\Hom(P_\CS(S)[i],N)\not=0$.
It follows that the right orthogonal category of
$\{P_\CS(S)[i]\}_{S\in\CS,i\in\BZ}$ is zero. Since the
$P_\CS(S)$ are perfect, it follows that
$\bigoplus_S P_\CS(S)$ generates the category of perfect complexes of
$A$-modules as a triangulated category closed under
taking direct summands \cite[Lemma 2.2]{Nee}. The functor
$\Hom^\bullet_A(\bigoplus_S P_\CS(S),-)$ gives
an equivalence $D^b(A)\iso D^b(B)$ \cite[Theorem 4.3]{Ke}.

\smallskip
Let $C=\bigoplus_{S\in\CS}P_\CS(S)$ and $N={^tH}^0(C)$.
We have a distinguished triangle
$\tau^{<0}C\to C\to N\rightsquigarrow$.
We have $\Hom(\tau^{<0}C,N[i])=0$ for $i\le 0$.
We deduce that the canonical morphism
$\Hom(N,N)\to \Hom(C,N)$ is an isomorphism.
We have $\Hom(C,(\tau^{<0}C)[i])=0$ for $i\ge 0$
since $\tau^{<0}C$ is filtered by objects
in $\CS[d]$, $d>0$ (cf Proposition \ref{construction}).
It follows that the canonical morphism
$\Hom(C,C)\to \Hom(C,N)$ is an isomorphism.

This shows that the canonical morphism
$\End(C)\to \End({^tH}^0(C))$ is an isomorphism.
By Lemma \ref{projinj}, ${^tH}^0(C)$
is a progenerator for $\CA$. So
$H^0(B)\mMod\simeq \CA$.

\smallskip
Note that $H^i(B)=0$ for $i\ll 0$ because
$\bigoplus_S P_\CS(S)$ is bounded.
Since $P_\CS(S)$ is filtered by objects in $\CS[d]$
with $d\ge 0$, it follows from Proposition \ref{construction}
that $\Hom(P_\CS(T),P_\CS(S)[i])=0$ for $i>0$. So,
$H^i(B)=0$ for $i>0$.
\end{proof}

The following proposition is clear.

\begin{prop}
Let $B$ be a dg-algebra with $H^i(B)=0$ for $i>0$ and for $i\ll 0$.
Let $C$ be the sub-dg-algebra of $B$ given by $C^i=B^i$ for $i<0$,
$C^0=\ker d^0$ and $C^i=0$ for $i>0$. Then
the restriction $D(B)\to D(C)$ is an equivalence.

Let $\CS$ be a complete set of representatives of isomorphism
classes of simple $H^0(B)$-modules (viewed as dg-$C$-modules).
Then $\CS$ satisfies
Hypothesis \ref{hyp2}. Furthermore, $\CA\simeq H^0(B)\mMod$.
\end{prop}

So we have a bijection between
\begin{itemize}
\item the sets $\CS$ (up to isomorphism) satisfying Hypothesis \ref{hyp2}
\item the equivalences $D^b(B)\iso D^b(A)$ where $B$ is a dg-algebra
with $H^i(B)=0$ for $i>0$ and for $i\ll 0$ and where $B$ is well-defined
up to quasi-isomorphism and the equivalence is taken modulo
self-equivalences of $D^b(B)$ that fix the isomorphism classes of simple
$H^0(B)$-modules.
\end{itemize}

\medskip

We recover a result of Al-Nofayee \cite[Theorem 4]{Al}~:
\begin{prop}
\label{selfinjder}
Assume $A$ is self-injective with Nakayama functor
$\nu$.
The following are equivalent
\begin{itemize}
\item
$H^i(B)=0$ for $i\not=0$
\item
$\nu(\CS)=\CS$ (up to isomorphism).
\end{itemize}
\end{prop}

\begin{proof}
Note that $\CS$ is stable under $\nu$ if and only if
$\{P_\CS(S)\}_{S\in\CS}$ is stable under $\nu$ (up to isomorphism).
Given $S,T\in\CS$ and $i\in\BZ$, we have
$$\Hom_{D^b(A)}(P_\CS(S),P_\CS(T)[i])^*\simeq
\Hom_{D^b(A)}(P_\CS(T),\nu(P_\CS(S))[-i]).$$

If $\CS$ is stable under $\nu$, then
$\Hom_{D^b(A)}(P_\CS(T),\nu(P_\CS(S))[-i])=0$ for $i>0$, hence
$H^{<0}(B)=0$.

Assume now $H^{<0}(B)=0$. Then, viewed as an object of $D^b(B)$,
$\nu(P_\CS(S))$ is concentrated in degree $0$. Since it is perfect, it is
isomorphic to a projective indecomposable module, hence to
$P_\CS(S')$ for some $S'\in\CS$. So, $\CS$ is stable under $\nu$.
\end{proof}

We recover now the main result of \cite{AlRi}:
\begin{cor}
Let $A$ be a self-injective algebra and $B$ an algebra derived equivalent
to $A$. Then $B$ is self-injective.
\end{cor}

From Proposition \ref{selfinjder},
we recover \cite[Theorem 5.1]{Ri}~:
\begin{thm}
\label{symmetricderived}
If $A$ is symmetric then $H^i(B)=0$ for $i\not=0$, \ie,
there is an equivalence $D^b(\CA)\iso D^b(A)$ where $\CS$ is the set of
images of the simple objects of $\CA$.
\end{thm}

\begin{rem}
Theorem \ref{symmetricderived} does not hold in general for a
self-injective algebra. Take $A=k[\eps]/(\eps^2)\rtimes \mu_2$, where 
$\mu_2=\{\pm 1\}$
acts on $k[\eps]/(\eps^2)$ by multiplication on $\eps$. Assume $k$ does not
have characteristic $2$. This is a self-injective algebra which is not
symmetric. The
Nakayama functor swaps the two simple $A$-modules $U$ and $V$.

Let $P_U$ (resp. $P_V$) be a projective cover of $U$ (resp. $V$).
Take $S=U$ and $T=P_U[1]$.
Then, the set $\CS=\{S,T\}$ satisfies Hypothesis \ref{hyp2}.
We have $I_\CS(T)\simeq T$ and $I_\CS(S)\simeq 0\to P_U\to P_V\to 0$,
a complex with homology $V$ in degree $0$ and $-1$.

The dg-algebra $B$ has homology $H^0(B)$ isomorphic
to the path algebra of the quiver
$\xymatrix{
\bullet\ar[r]&\bullet}$,
$H^{-1}(B)=k$ and $H^i(B)=0$ for $i\not=0,-1$.

The derived category of the hereditary algebra $H^0(B)$ is not
equivalent to $D^b(A)$.
\end{rem}

\subsection{Graded of an abelian category}
Let $\CA$ be an abelian $k$-linear artinian and noetherian
category with finitely many
simple objects up to isomorphism and $\CS$ a complete set
of representatives of isomorphism classes of simple objects. We assume
$\CA$ is split, \ie, endomorphism rings of simple objects are isomorphic to
$k$.
Let $\CT=D^b(\CA)$.

Let $\gr\CA$ be the category with objects the objects of $\CA$ and where
$\Hom_{\gr\CA}(M,N)$ is the graded vector space associated to the
filtration of $\Hom_\CA(M,N)$ given by
$\Hom_\CA(M,N)^i=\{f | \im f\subseteq \rad^i N\}$.

\smallskip
Given $M$ in $\CA$, let $M_i=\rad^i M$, $f_i:M_i\to M_{i-1}$ the
inclusion, $N_0=M/M_1$ and $\eps_0:M\to M/M_1$ the projection.
This defines an object of $\CF$.

We obtain a functor $\gr\CA\to \CF$.

\begin{prop}
The canonical functor $\gr\CA\to \CF$ is an equivalence.
\end{prop}

\begin{proof}
The image of $\Hom_\CA(N,N')$ in $\Hom_\CA(N,N'_0)$ is isomorphic
to the quotient of $\Hom_\CA(N,N')$ by $\Hom_\CA(N,\rad N')$ and it
follows that the functor is fully faithful.

Let us show that it is essentially surjective.
Let $M\in\CF$. Let $r\ge 0$ such that $M_{r+1}=0$.
Then, $M_r\iso N_r$ has homology concentrated in degree $0$ and is
semi-simple. By induction on $-i$,
it follows from the distinguished triangle
$M_{i+1}\to M_i\to N_i\rightsquigarrow$
that $M_i$ has homology concentrated in degree $0$.

Note that we have an exact sequence
$0\to H^0M_{i+1}\to H^0M_i\to H^0N_i\to 0$.
Since the canonical map
$\Hom(H^0N_i,S)\to \Hom(H^0M_i,S)$ is bijective for any simple $S$,
it follows that $H^0N_i$ is the largest
semi-simple quotient of $H^0M_i$.
So, $M_i\iso \rad^i M_0$ and $M$ comes from an object of $\CA$.
\end{proof}

\section{Simple generators for stable categories}
\subsection{From equivalences}
Let $k$ be a field and $A$ a split self-injective
$k$-algebra with no projective simple module. 

Let $B$ be another split self-injective $k$-algebra with no projective
simple module, and
let $F:B\mstab\iso A\mstab$ be an equivalence of triangulated categories.
Let $\CS'$ be a complete set of representatives of isomorphism classes
of simple $B$-modules. For $L\in\CS'$, let
$L'$ be an indecomposable $A$-module isomorphic to $F(L)$ in $A\mstab$.
Let $\CS=\{L'\}_{L\in\CS'}$. Then,
\begin{itemize}
\item[(i)]
$\Hom_{A\mstab}(S,T)=k^{\delta_{S,T}}$ for $S,T\in\CS$
\item[(ii)]
Every object $M$ of $A\mstab$ has a filtration
$0=M_r\to M_{r-1}\to \cdots\to M_1\to M_0=M$ such that the cone
of $M_i\to M_{i-1}$ is isomorphic to an object of $\CS$.
\end{itemize}

Note that (ii) is equivalent to
\begin{itemize}
\item[(ii')]
Given $M$ in $A\mMod$, there is a projective module $P$ such that
$M\oplus P$ has a filtration
$0=N_r\subset N_{r-1}\subset \cdots\subset N_1\subset N_0=M\oplus P$
with the property that $N_i/N_{i-1}$ is isomorphic (in $A\mMod$) to
an object of $\CS$.
\end{itemize}

Linckelmann has shown the following \cite[Theorem 2.1 (iii)]{Li}~:

\begin{prop}
Assume that $F$ is induced by an exact functor $B\mMod\to A\mMod$.
If $\CS$ consists of simple modules, then there is a direct summand of
$F$ that is an equivalence $B\mMod\iso A\mMod$.
\end{prop}

We deduce~:
\begin{cor}
Let $B_1$, $B_2$ be split self-injective algebras with no
projective simple modules and $G_i:B_i\mMod\to A\mMod$ exact functors
inducing stable equivalences. Assume $\CS_1=\CS_2$ (up to isomorphism).
Then, $B_1$ and $B_2$ are Morita equivalent.
\end{cor}

So, if we assume in addition that $F$ comes from an exact functor $G$ between
module categories, then $B$ is determined by $\CS$, up to Morita equivalence.

The functor $G$ is isomorphic to $X\otimes_B -$ where $X$ is
an $(A,B)$-bimodule. We can (and will) choose $G$ so that $X$ has no  non-zero
projective direct summand. Then, $G(L)$ is indecomposable for
$L$ simple \cite[Theorem 2.1 (ii)]{Li}, so $\CS=\{G(L)\}_{L\in\CS'}$, up to
isomorphism.

\begin{prop}
An $A$-module $M$ is in the image of $G$ if and only if there is
a filtration
$0=M_r\subset M_{r-1}\subset \cdots\subset M_1\subset M_0=M$ such that
$M_i/M_{i-1}$ is isomorphic to an object of $\CS$.
\end{prop}

\begin{proof}
Take $L$ a $B$-module. Then the image by $G$ of
a filtration of $L$ whose successive quotients are simple provides
a filtration as required.

Conversely, we proceed by induction on $r$.
We have an exact sequence $0\to G(N)\to M\to G(L)\to 0$ and
a corresponding element $\zeta\in\Ext^1_A(G(L),G(N))$. We
have an isomorphism $\Ext^1_B(L,N)\iso \Ext^1_A(G(L),G(N))$ and
we take $\zeta'$ to be the inverse image of $\zeta$ under this isomorphism.
This gives an exact sequence $0\to N\to M'\to L\to 0$, and hence an
exact sequence $0\to G(N)\to G(M')\to G(L)\to 0$ with
class $\zeta$. It follows that $M\simeq G(M')$ and we are done.
\end{proof}

\subsection{Filtrable objects}
\subsubsection{}

Given two $A$-modules $M$ and $N$, we write $M\sim N$ to denote the
existence of an isomorphism between $M$ and $N$ in $A\mstab$.
Given $f,g\in\Hom_A(M,N)$, we write $f\sim g$ if $f-g$ is a projective map.

\begin{lemma}
\label{stableker}
Let
$f,f':M\to N$ be two surjective maps with $f\sim g$. Then there is
$\sigma\in\Aut_A(M)$ with $f'=f\sigma$ and $\sigma\sim\id_M$.

Similarly, let
$f,f':N\to M$ be two injective maps with $f\sim g$. Then there is
$\sigma\in\Aut_A(M)$ with $f'=\sigma f$ and $\sigma\sim\id_M$.

\end{lemma}

\begin{proof}
Let $L=\ker f$ and $L'=\ker f'$. Let $L=L_0\oplus P$
and $L'=L'_0\oplus P'$ with $P$, $P'$ projective and $L_0$, $L'_0$
without non-zero projective direct summands. We have an isomorphism
$\bar{\alpha}_0\in\Hom_{A\mstab}(L_0,L'_0)$
in $A\mstab$ giving rise to an isomorphism of distinguished
triangles in $A\mstab$
$$\xymatrix{
L_0\ar[d]_{\bar{\alpha}_0}^\sim \ar[r] & M\ar[r]\ar@{=}[d] & N\ar[r]\ar@{=}[d]
 & \Omega^{-1}L_0\ar[d]_{\Omega^{-1}(\bar{\alpha}_0)}^\sim \\
L'_0 \ar[r] & M\ar[r] & N\ar[r] &  \Omega^{-1}L'_0
}$$
Let $\alpha_0\in\Hom_A(L_0,L'_0)$ lifting $\bar{\alpha}_0$. This is an
isomorphism. There is now a commutative diagram of $A$-modules,
where the exact rows come from the elements of
$\Ext^1_A(N,L_0)$ and $\Ext^1_A(N,L'_0)$ defined above~:
$$\xymatrix{
0\ar[r] & L_0 \ar[r] \ar[d]_{\alpha_0}^\sim & M_0 \ar[r]
\ar@{.>}[d]_{\sigma_0}^\sim &
 N\ar[r]\ar@{=}[d] & 0 \\
0\ar[r] & L'_0 \ar[r] & M'_0 \ar[r] & N\ar[r] & 0
}$$
We have $M\simeq M_0\oplus P\simeq M'_0\oplus P'$, hence $P\simeq P'$.
Let $\alpha:L\iso L'$ extending $\alpha_0$. Then there is
$\sigma:M\iso M$ making the following diagram commute
$$\xymatrix{
0\ar[r] & L \ar[r] \ar[d]_\alpha^\sim & M \ar[r]
\ar@{.>}[d]_{\sigma}^\sim &
 N\ar[r]\ar@{=}[d] & 0 \\
0\ar[r] & L \ar[r] & M \ar[r] & N\ar[r] & 0
}$$
and we are done.

%
%

\smallskip
The second part of the lemma has a similar proof --- it can also be deduced
from the first part by duality.
\end{proof}

\subsubsection{}
\begin{hyp}
\label{hyp3}
Let $\CS$ be a finite set of indecomposable finitely generated
$A$-modules such that
$\Hom_{A\mstab}(S,T)=k^{\delta_{S,T}}$ for $S,T\in \CS$.
\end{hyp}

An {\em $\CS$-filtration} for an $A$-module $M$
is a filtration $0=M_r\subseteq M_{r-1}\subseteq\cdots\subseteq M_0=M$
such that $\Mbar_i=M_i/M_{i+1}$ is in $\add(\CS)$ for $0\le i\le r-1$.

We say that $M$ is {\em filtrable} if it admits an $\CS$-filtration.

\begin{lemma}
\label{nonzeromap}
Let $M$ be a non-projective filtrable $A$-module. Then there is
$S\in\CS$ such that $\Hom_{A\mstab}(M,S)\not=0$ (resp. such that
$\Hom_{A\mstab}(S,M)\not=0$).
\end{lemma}

\begin{proof}
Assume $\Hom_{A\mstab}(M,S)=0$ for all $S\in\CS$. Since $M$ is
filtrable, it follows that $\End_{A\mstab}(M)=0$, and hence $M$ is
projective, which is not true.
The second case is similar.
\end{proof}

\begin{lemma}
\label{nonprojsurj}
Let $M$ be a filtrable module and $S\in \CS$. Given $f:M\to S$ non-projective,
there is $g:M\to S$ surjective with filtrable kernel such that 
$f\sim g$. Similarly, given $f:S\to M$ non-projective, 
there is $g:S\to M$ injective with filtrable cokernel such that 
$f\sim g$.
\end{lemma}

\begin{proof}
We proceed by induction on the number of terms in a filtration of $M$.
The result is clear if $M\in\CS$.

Let $0\to N\Rarr{\alpha} M\Rarr{\beta}T\to 0$ be an exact sequence with
$T\in\CS$ and $N$ filtrable.

\smallskip
Assume first $f\alpha:N\to S$ is projective. Then there is $p:M\to S$
projective and $g:T\to S$ with $f-p=g\beta$. Since $g$ is not
projective, it is an isomorphism. Consequently, $f-p$ is surjective
and its kernel is isomorphic to $N$ by Lemma \ref{stableker}, so we are done.

\smallskip
Assume now $f\alpha:N\to S$ is not projective.
By induction, there is $q:N\to S$ projective such that $f\alpha+q$ is
surjective with filtrable kernel $N'$.
Since $\alpha:N\to M$ is
injective,
there is a projective map $p:M\to S$ with $q=p\alpha$.
Now, we have an exact sequence
$0\to N/N'\Rarr{\bar{\alpha}} M/\alpha(N')\to T\to 0$
and a non-projective surjection
$f+p:M/\alpha(N')\to S$. Since $(f+p)\bar{\alpha}:N/N'\iso S$ is
an isomorphism, it follows that the kernel of the map
$M/\alpha(N')\to S$ is isomorphic to $T$. Since $N'$ is filtrable,
it follows that $\ker(f+p)$ is filtrable and we are done.
The second assertion follows by duality.
\end{proof}

From Lemmas \ref{stableker} and \ref{nonprojsurj}, we deduce~:

\begin{lemma}
\label{filtrationkercoker}
Let $S\in\CS$ and let $M$ be a filtrable module.

If $f:M\to S$ be a surjective and non-projective map, then $\ker f$
is filtrable.

Similarly, if $g:S\to M$ is injective and non-projective, then
$\coker g$ is filtrable.
\end{lemma}

From Lemmas \ref{nonzeromap} and \ref{nonprojsurj}, we deduce~:
\begin{lemma}
Let $M$ be filtrable non-projective. Then there is a
submodule $S$ of $M$, with $S\in\CS$, 
such that $M/S$ is filtrable and the inclusion $S\to M$ is not
projective. Similarly,
there is a filtrable submodule $N$ of $M$
such that $M/N\in\CS$ and $M\to M/N$ is not projective.
\end{lemma}

\begin{prop}
\label{decomposition}
Let $M$ be an $A$-module with a decomposition $M\sim M'_1\oplus M'_2$
in the stable category. If $M$ is filtrable then there is
a decomposition $M=M_1\oplus M_2$ such that $M_i$ is filtrable and
$M_i\sim M'_i$
\end{prop}

\begin{proof}
We can assume $M$ is not projective, for otherwise the proposition is
trivial.
We prove the proposition by induction on the dimension of $M$.

Let $M=T_1\oplus T_2\oplus P$ with $P$ projective, $T_i$ without
non-zero projective direct summand and $T_i\sim M'_i$. 
Denote by $\pi:M\to T_1$ the projection.

By Lemma \ref{nonzeromap}, there is $S\in\CS$ such that
$\Hom_{A\mstab}(M,S)\not=0$. Hence, $\Hom_{A\mstab}(T_i,S)\not=0$ for
$i=1$ or $i=2$. Assume for instance $i=1$. Pick a non-projective map
$\alpha:T_1\to S$. So, $\alpha\pi:M\to S$ is not projective. By
Lemma \ref{nonprojsurj}, there is a surjective map $\beta:M\to S$ with
$\beta\sim \alpha\pi$ and $N=\ker\beta$ filtrable. Then
$N\sim L\oplus T_2$, where $L$ is the kernel of $\alpha+p:
T_1\oplus P_S\to S$ and $p:P_S\to S$ is a projective cover of $S$.
By induction, we have $N=N_1\oplus N_2$ with $N_i$ filtrable and
$N_1\sim L$, $N_2\sim T_2$. Now, the map $S\to L[1]$ gives a map
$S\to N_1[1]$ (in $A\mstab$).
Let $M_1$ be the extension of $S$ by $N_1$ corresponding
to that map. Then $M\simeq M_1\oplus N_2$, the modules $M_1$ and $N_2$ are
filtrable, $M_1\sim M'_1$, and $N_2\sim M'_2$.
\end{proof}

Let $M$ be a filtrable module. We say that $M$ {\em has no projective
remainder} if there is no direct sum decomposition $M=N\oplus P$ with
$P\neq0$ projective and $N$ filtrable.

\begin{lemma}
\label{caracfilt}
Let $M$ be a filtrable module with no projective remainder
and let $S\in\CS$.

For $f:M\to S$ surjective,
$\ker f$ is filtrable if and only if $f$ is non-projective.

For $f:S\to M$ injective,
$\coker f$ is filtrable if and only if $f$ is non-projective.
\end{lemma}

\begin{proof}
Assume $f$ is projective. Then there is a decomposition
$M=N\oplus P$ and $f=(0,g)$ with $P$ projective. Now,
$\ker f=N\oplus \ker g$. If $\ker f$ is filtrable, then it follows from
Lemma \ref{decomposition} that $M$ has a non-zero projective submodule whose
quotient is filtrable.

The converse is given by Lemma \ref{filtrationkercoker}.
The second part of the Lemma has a similar proof.
\end{proof}

\begin{lemma}
\label{dirfilt}
Let $M=M_0\oplus M_1$ with $M$ and $M_0$ filtrable and such that $M_0$ has no
projective remainder. Then $M_1$ is filtrable.
\end{lemma}

\begin{proof}
We proceed by induction on $\dim M_0$ --- the result is clear for
$M_0=0$.
Assume $M_0\not=0$.
Let $f:M_0\to S$ be a surjection with $S\in\CS$ and $\ker f$ filtrable. By
Lemma \ref{caracfilt}, $f$ is not projective.
Then $f':M\Rarr{\can} M_0\Rarr{f} S$ is a non-projective surjection.
By Lemma \ref{filtrationkercoker}, $\ker f'$ is filtrable.
We have $\ker f'=\ker f\oplus M_1$ and we are done.
\end{proof}

\subsubsection{}
We now turn to filtrations by objects in $\add(\CS)$.

\begin{lemma}
\label{minimalsurj}
Let $M$ be a filtrable module and $N$ a filtrable submodule of $M$
such that $M/N\in\add\CS$. Then, $N$ is minimal with these
properties if and only if $N$ has no projective remainder and
the canonical map
$\Hom_{A\mstab}(M/N,S)\to \Hom_{A\mstab}(M,S)$ is surjective for every
$S\in\CS$.
\end{lemma}

\begin{proof}
Let $N$ be a minimal filtrable submodule of $M$ such that
$M/N\in\add\CS$. Denote by $i:N\to M$ the injection and
$p:M\to M/N$ the quotient map.

Let $S\in\CS$.
Fix $f_1,\ldots,f_r:M/N\to S$ such that 
$\sum_i f_i:M/N\to S^r$ is surjective and $\ker\sum_i f_i$ has no direct
summand isomorphic to $S$.
Let $T$ be the subspace of $\Hom_{A\mstab}(M,S)$ generated by
$f_1p,\ldots,f_rp$. Assume this is a proper subspace, so there is
$f':M\to S$ whose image in $\Hom_{A\mstab}(M,S)$ is not in $T$.
Then $f'i:N\to S$ is not projective, hence there is a projective
map $q:N\to S$ such that $f'i+q$ is surjective and has
filtrable kernel $N'$ (Lemma \ref{nonprojsurj}). There
is $q':M\to S$ projective such that $q=q'i$. Now,
$M/N'\simeq M/N\oplus S$ and this contradicts the minimality of $N$.
It follows that the canonical map
$\Hom_{A\mstab}(M/N,S)\to \Hom_{A\mstab}(M,S)$ is surjective.
Assume $N=N'\oplus P$ with $N'$ filtrable with no projective remainder
and $P$ projective. By Lemma \ref{dirfilt}, $P$ is filtrable. We have 
$M/N'\simeq M/N\oplus P$. Since $M/N$ is a maximal quotient of $M$ in
$\add(\CS)$ and $P$ is filtrable, it follows that $P=0$.

Conversely, take $f:N\to S$ surjective with filtrable kernel such that
the extension of $M/N$ by $S$ splits. Then $f$ lifts to
$M\to S$ and it is not projective by Lemma \ref{caracfilt}. This contradicts
the surjectivity of $\Hom_{A\mstab}(M/N,S)\to \Hom_{A\mstab}(M,S)$.
Consequently, $N$ is minimal.
\end{proof}

\begin{lemma}
\label{filtifinj}
Let $M$ be a filtrable $A$-module with no projective remainder.

Let $f:M\to L$ be a surjection with $L\in\add\CS$.
Then $\ker f$ is filtrable if and only if
the canonical map $\Hom_{A\mstab}(L,S)\to\Hom_{A\mstab}(M,S)$ is
injective for all $S\in\CS$.
\end{lemma}

\begin{proof}
Note that the canonical map $\Hom_{A\mstab}(L,S)\to\Hom_{A\mstab}(M,S)$ is
injective if and only if, given $p:L\to S$ surjective
with $S\in\CS$, $pf$ is not projective.

Assume $\ker f$ is filtrable. Let $p:L\to S$ be a surjective map
with $S\in\CS$. Then $\ker pf$ is filtrable, hence
$pf$ is not projective (Lemma \ref{caracfilt}).

Let us now prove the converse by induction on the dimension of $M$.
Assume that given $p:L\to S$ surjective
with $S\in\CS$, then $pf$ is not projective.
Pick $p:L\to S$ surjective and let $L'=\ker p$. Let
$M'=\ker pf$. Then $f$ induces a surjection $f':M'\to L'$ and we have
$L'\in\add\CS$ (since $p$ is split). Let $p':L'\to T$ be a surjective
map with $T\in\CS$. Fix a left inverse $\sigma:L\to L'$ to the inclusion
$L'\to L$.
$$\xymatrix{
& & 0\ar[d] & 0\ar[d] & T \\
0 \ar[r] & \ker f' \ar@{=}[d] \ar[r] & M'\ar[r]^{f'}\ar[d] &
 L'\ar[d] \ar@{>>}[ur]^{p'} \ar[r] & 0 \\
0 \ar[r] & \ker f  \ar[r] & M\ar[r]_{f}\ar[d] &
 L\ar[d]^p \ar@{.>}@/_/[u]_{\sigma} \ar[r] & 0 \\
&&S\ar@{=}[r]\ar[d] & S\ar[d] \\
&&0&0
}$$

If $S\not=T$, then $\Hom_{A\mstab}(S,T)=0$, and hence
$p'\sigma f$ doesn't factor through $S$ in the stable category.
On the other hand, if $S=T$ then
$pf$ and $p'\sigma f$ define linearly independent elements
of $\Hom_{A\mstab}(M,S)$. Consequently, $p'\sigma f$ doesn't factor
through $S$ in the stable category.
It follows that $p'f'$ is not projective. By Lemma \ref{filtrationkercoker},
$M'$ is filtrable. By induction, it follows that $\ker f'$ is filtrable
and we are done.
\end{proof}

\begin{prop}
\label{head}
Let $M$ be a filtrable $A$-module with no projective remainder.

Let $N$ be a minimal filtrable submodule of $M$ such that
$M/N\in\add\CS$. Then there is an isomorphism
$$M/N\iso \bigoplus_{S\in\CS}S\otimes \Hom_{A\mstab}(M,S)$$
that induces the canonical map
$M\to \bigoplus_{S\in\CS}S\otimes \Hom_{A\mstab}(M,S)$
in the stable category.

Given $\tau\in\Aut(N)$ such that $\tau\sim\id_N$, there is
$\sigma\in\Aut(M)$ with $\sigma\sim\id_M$ and $\sigma_{|N}=\tau$.

Let $N'$ be a minimal filtrable submodule of $M$ such that
$M/N'\in\add\CS$. Then there is $\sigma\in\Aut(M)$ such that
$N'=\sigma(N)$ and $\sigma\sim \id_M$.
\end{prop}

\begin{proof}
The first part of the proposition follows from Lemmas \ref{minimalsurj} and
\ref{filtifinj}.

\smallskip
Let $\tau\in\Aut(N)$ such that $\tau=\id_N+p$ with $p:N\to N$
projective. Then there is a projective map $q:M\to N$ with $p=qi$.
Let $\sigma=\id_M+q$. Then $\sigma_{|N}=\tau$. Now, 
we have a commutative diagram
$$\xymatrix{
0\ar[r] & N\ar[r]\ar[d]_\tau^\sim & M\ar[r]\ar@{.>}[d]_\sigma &
 M/N\ar[d]_{\id}\ar[r] & 0 \\
0\ar[r] & N\ar[r] & M\ar[r] & M/N\ar[r] & 0
}$$
and hence $\sigma$ is an automorphism of $M$.

\smallskip
Let $N'$ be a minimal filtrable submodule of $M$ such that
$M/N'\in\add\CS$. Then we have shown that $M/N\iso M/N'$ and
that via such an isomorphism, the maps
$M\to M/N$ and $M\to M/N'$ are stably equal. Now, Lemma \ref{stableker}
shows there is $\sigma\in\Aut(M)$ with $N'=\sigma(N)$ and $\sigma\sim \id_M$.
\end{proof}

Let $M$ be filtrable. An {\em $\CS$-radical filtration} of
$M$ is a filtration
$0=M_r\subseteq M_{r-1}\subseteq\cdots\subseteq M_0=M$ such
that $M_i$ is a minimal filtrable submodule of $M_{i-1}$ with
$M_{i-1}/M_i\in\add\CS$.

\begin{prop}
\label{series}
Let $M$ be a filtrable $A$-module with no projective remainder.
Let $0=M_r\subseteq M_{r-1}\subseteq\cdots\subseteq M_0=M$
and $0=M'_{r'}\subseteq M'_{r'-1}\subseteq\cdots\subseteq M'_0=M$ be two
$\CS$-radical filtrations of $M$. Then, $r=r'$ and
there is an
automorphism of $M$ that swaps the two filtrations and that is stably
the identity.
\end{prop}

\begin{proof}
We prove this lemma by induction on the dimension of $M$.
By Proposition \ref{head}, there is $\sigma\in\Aut(M)$ such that
$\sigma(M'_1)=M_1$ and $\sigma\sim\id_M$. Now, by induction, we have
$r=r'$ and
there is $\tau\in\Aut(M_1)$ such that $\tau\sigma(M'_i)=M_i$
for $i>0$ and $\tau\sim \id_{M_1}$.
By Proposition \ref{head}, there is
$\tau'\in\Aut(M)$ such that $\tau'_{|M_1}=\tau$ and
$\tau'\sim\id_M$. Now, $\tau'\sigma$ sends $M'_i$ onto $M_i$.
\end{proof}

\begin{rem}
A filtrable projective module can have two $\CS$-radical filtrations with 
non-isomorphic layers.

Consider $A=k\GA_4$, the group algebra of
the alternating group of degree $4$ and assume $k$ has characteristic $2$ and
contains a cubic root of $1$. Let $B$ be the principal block of $k\GA_5$.
Then, the restriction functor is a stable equivalence between $B$ and $A$.
Let $\CS$ be the set of images of the simple $B$-modules.
Denote by $k$ the trivial $A$-module and by $k_+$, $k_-$ the non-trivial
simple $A$-modules. Then $\CS=\{k,S_+,S_-\}$ where $S_\eps$ is a non-trivial
extension of $\eps$ by $-\eps$. Let $P$ and $P'$ be the two projective
indecomposable $B$-modules that don't have $k$ as a quotient. Then
$\Res_{\GA_4}P\simeq \Res_{\GA_4}P'$. This projective module has
two $\CS$-radical filtrations with non-isomorphic layers~:
one coming from the radical filtration of
$P$ and one coming from the radical filtration of $P'$.
\end{rem}

While $\CS$-radical filtrations are not unique in general
for filtrable modules with a projective remainder, there are some cases
where uniqueness still holds~:
\begin{prop}
Assume $A$ is a symmetric algebra.
Let $0\to S\to M\to T\to 0$ and
$0\to S'\to M\to T'\to 0$ be two exact sequences with
$S,S',T,T'\in\CS$. Assume that the sequences don't both split.
Then there is an automorphism of $M$ swapping the two exact sequences.
\end{prop}

\begin{proof}
If $M$ is non-projective, then this is a consequence of Proposition \ref{head}.

Assume $M$ is projective.
Since $A$ is symmetric, we have a non-projective map
$T\simeq \Omega^{-1}S\to S$. It follows that $S=T$.
Similarly, $T'=S'$. We have exact sequences
$$0\to \Hom(S',S)\to \Hom(S',M)\to \Hom(S',S)\to \Ext^1(S',S)\to 0$$
$$0\to \Hom(S',S')\to \Hom(S',M)\to \Hom(S',S')\to \Ext^1(S',S')\to 0$$
We have $\Omega^{-1}S'\simeq S'$, and hence $\dim\Ext^1(S',S')=1$.
Consequently, $\dim \Hom(S',M)$ is an odd integer. It follows that
$\Ext^1(S',S)\not=0$, hence $\Hom_{A\mstab}(S',S)\not=0$, so $S'=S$
and we are done by Lemma \ref{stableker}.
\end{proof}

\begin{lemma}
\label{projinfilt}
Let $0=M_r\subset M_{r-1}\subset\cdots\subset M_0=M$ be a filtration
of $M$ with $M_{i-1}/M_i\in\add\CS$.
\begin{itemize}
\item[(i)]
If $M$ has no projective
remainder, then $M_i$ has no projective remainder, for all $i$.
\item[(ii)]
If the filtration is an $\CS$-radical filtration, then $M_i$ has no
projective remainder for $i\ge 1$.
\end{itemize}
\end{lemma}

\begin{proof}
Consider an exact sequence $0\to N\oplus P\to M\to L\to 0$ of filtrable modules
with $P$ projective and $N$ filtrable. Then there is an extension $M'$ of
$L$ by $N$ such that $M=M'\oplus P$ and $M'$ is filtrable.
The first part of the lemma follows.

Assume now the filtration is an $\CS$-radical filtration.
Assume for some $i\ge 1$, we have
$M_i=N\oplus P$ with $N$ filtrable with no projective
remainder and $P$ projective and filtrable (Lemma \ref{dirfilt}).
Then, $M=M'\oplus P$ with $P$ filtrable by (i).
There is an exact sequence
$0\to L\to P\to S\to 0$ with $S\in\CS$ and $L$ filtrable. Now,
the canonical surjection $M'\oplus P\to M/M_1\oplus S$ has filtrable
kernel and this contradicts the minimality of $M_1$.
\end{proof}

\begin{prop}
Let $M_1$ and $M_2$ be two filtrable $A$-modules with no projective
remainder.
If $M_1\sim M_2$, then $M_1\simeq M_2$.
\end{prop}

\begin{proof}
We prove the proposition by induction on 
$\min(\dim M_1,\dim M_2)$.
Fix an isomorphism $\phi$ from $M_2$ to $M_1$ in the stable category.
Let $X=\bigoplus_{S\in\CS}S\otimes \Hom_{A\mstab}(M_1,S)$ and
$g_1\in\Hom_{A\mstab}(M_1,X)$ be the canonical map.
Let $g_2=g_1\phi$.
By Propositions \ref{head} and \ref{series}, there are exact sequences
$$0\to N_1\to M_1\Rarr{f_1} X\to 0\ \textrm{ and }\ 
0\to N_2\to M_2\Rarr{f_2} X\to 0$$
with the image of $f_i$ in the stable category equal to $g_i$.
So, there is an isomorphism from $N_2$ to $N_1$ in the stable
category compatible with $\phi$.
By Lemma \ref{projinfilt}, $N_1$ and $N_2$ have no projective remainder.
By induction, we deduce that there is
an isomorphism $N_2\iso N_1$ lifting the stable isomorphism.
So, $M_1$ and $M_2$ are extensions of isomorphic modules,
with the same class in $\Ext^1$, hence are isomorphic.
\end{proof}

\subsection{Generators and reconstruction}
\subsubsection{}
We assume from now on that

\begin{hyp}
\label{hyp4}
$\CS$ satisfies Hypothesis \ref{hyp3} and
given $M\in A\mMod$, there is a
projective $A$-module $P$ such that $M\oplus P$ is filtrable.
\end{hyp}

\begin{prop}
Let $S\in\CS$. Let $P_S\to S$ be a projective cover of $S$ and
$P$ minimal projective such that $\Omega S\oplus P$ is filtrable.
Let $0=M_r\subseteq M_{r-1}\subseteq\cdots\subseteq M_1\subseteq M_0=
\Omega S\oplus P$ be an $\CS$-radical filtration.

Then
$0=M_r\subseteq M_{r-1}\subseteq\cdots\subseteq M_1\subseteq M_0\subseteq
P_S\oplus P$ is an $\CS$-radical filtration.

If $A$ is symmetric, then $M_{r-1}\simeq S$.
\end{prop}

\begin{proof}
Let $f_1:P_S\to S$ be a surjective map and $f=(f_1,0):P_S\oplus P\to S$.
Let $T\in\CS$ and $g:P_S\oplus P\to T$ such that
we have an exact sequence
$0\to L\to P_S\oplus P\Rarr{f+g}S\oplus T\to 0$
with $L$ filtrable.

We have a commutative diagram
$$\xymatrix{
0\ar[r] & L\ar[r]\ar@{=}[d] & P_S\oplus P\ar[r]^{f+g} & S\oplus
T\ar[d]^{(0,\id)}\ar[r] & 0 \\
0\ar[r] & L\ar[r] & \Omega S\oplus P\ar@{^{(}->}[u]\ar[r] &T\ar[r] & 0 
}$$
The surjection $\Omega S\oplus P\to T$ is projective and has filtrable
kernel. From Lemma \ref{caracfilt}, we get a contradiction to the minimality of
$P$. It follows that $\Omega S\oplus P$ is a minimal submodule of
$P_S\oplus P$ such that the quotient is in $\add\CS$.

\smallskip
We have $\Hom_{A\mstab}(T,\Omega S)\simeq \Hom_{A\mstab}(S,T)^*$,
since $A$ is symmetric. Now, $\Hom_{A\mstab}(M_{r-1},\Omega S\oplus P)\not=0$
by Lemma \ref{caracfilt}. The second part of the proposition follows.
\end{proof}

Let $M$ and $N$ be two $A$-modules with 
filtrations
$0=M_r\subseteq M_{r-1}\subseteq\cdots\subseteq M_0=M$
$0=N_s\subseteq N_{s-1}\subseteq\cdots\subseteq N_0=N$.
Let $\Hom_A^f(M,N)$ be the subspace of $\Hom_A(M,N)$ of
filtered maps (\ie, those $g$ such that $g(M_i)\subseteq N_i$).
We put $\Mbar_i=M_i/M_{i+1}$. We denote
by $\phi_i$ the composition of canonical maps
$\phi_i:\Hom_A^f(M,N)\to \Hom_A(\Mbar_i,\Nbar_i)\to
\Hom_{A\mstab}(\Mbar_i,\Nbar_i)$.

We view $N'=N_i$ as a filtered module with the induced filtration
$0=N'_{s-i}\subseteq N'_{s-i-1}=N_{s-1}\subseteq\cdots\subseteq 
N'_1=N_{i+1}\subseteq N'_0=N'$.

\begin{lemma}
\label{indeplift}
Let $M$ be a filtrable $A$-module with an $\CS$-radical filtration 
and $N$ be a filtrable $A$-module with an $\CS$-filtration.
Let $f\in\Hom_A^f(M,N)$ with $\phi_0(f)=0$. Then
$\phi_i(f)=0$ for all $i$.
\end{lemma}

\begin{proof}
The map $\fbar_0:\Mbar_0\to\Nbar_0$ induced by $f$ is projective. So
there is a projective module $P$ and a commutative diagram
$$\xymatrix{
M \ar@{>>}[dd] && N\ar@{>>}[dd] \\
& P \ar@{.>}[ur]\ar[dr] \\
\Mbar_0\ar[rr]_{\fbar_0}\ar[ur] && \Nbar_0
}$$
Let $p$ be the composition $p:M\to \Mbar_0\to P\to N$. Then
$f-p\sim f$, $f-p$ and $f$ have the same restriction to $M_1$, and
$\overline{(f-p)}_0=0$. Consequently it is enough to prove the lemma in the
case where $\fbar_0=0$.

From now on, we assume $\fbar_0=0$.
Assume the map $\fbar_1:\Mbar_1\to\Nbar_1$ is not projective. So
there is $S\in\CS$ and a (split) surjection $g:\Nbar_1\to S$ such that
$g\fbar_1:\Mbar_1\to S$ is not projective. Let $s:S\to \Mbar_1$ be a
right inverse to $g$, and let $L$
be the kernel of $g\fbar_1$.

We have an exact sequence
$0\to L\to M/M_2\Rarr{(\can,gf)} \Mbar_0\oplus S\to 0$.
So the inverse image of $L$ in $M_1$ is a filtrable submodule of
$M$ with quotient isomorphic to $\Mbar_0\oplus S$. This contradicts the
fact that $M_1$ is a minimal filtrable submodule of $M$ such that
$M/M_1\in\add\CS$.
So $\fbar_1$ is projective; \ie, $\phi_1(f)=0$.

We now prove by induction that $\phi_i(f)=0$ for all $i$. Assume 
$\phi_d(f)=0$. Then, we apply the result above to
the filtered modules $M_d$ and $N_d$ to get $\phi_{d+1}(f)=0$.
\end{proof}

\subsubsection{}
We define a category $\CG$ as follows.

$\bullet$ Its objects are $A$-modules together with a fixed
$\CS$-radical filtration.

$\bullet$ We define
$\Hom_\CG(M,N)_i$ as the image of $\Hom_A^f(M,N_i)$ in
$\Hom_{A\mstab}(\Mbar_0,\Nbar_i)$. We put
$\Hom_\CG(M,N)=\oplus_i \Hom_\CG(M,N)_i$.

$\bullet$ Let $f\in \Hom_\CG(M,N)_i$ and $g\in \Hom_\CG(L,M)_j$.
Let $\tf:M\to N_i$ be a filtered map lifting $f$. It induces
a map $\phi_j(\tf)\in\Hom_{A\mstab}(\Mbar_j,\Nbar_{i+j})$
independent of the choice of $\tf$ (Lemma \ref{indeplift}).
We define the product $fg$ to be $\phi_j(\tf)\circ \phi_0(g)$.

\smallskip
Given $S\in\CS$, let $P_S\to S$ be a projective cover of $S$ and
$Q_S$ projective minimal such that $\Omega S\oplus Q_S$ is filtrable.
Fix a radical filtration of $P_S\oplus Q_S$ with first term
$\Omega S\oplus Q_S$.

Let $M=\oplus_{S\in\CS}(P_S\oplus Q_S)$. This comes with an $\CS$-radical
filtration.
We have constructed a $\BZ_{\ge 0}$-graded $k$-algebra $\End_\CG(M)$.

The following Lemma is clear.
\begin{lemma}
Let $\CS$ be a complete set of representatives of isomorphism
classes of simple $A$-modules. Then
we have an equivalence $\gr(A\mMod)\iso \CG$. If $A$ is basic,
then $\End_\CG(M)$ is
isomorphic to the graded algebra associated with the radical filtration
of $A$.
\end{lemma}

\smallskip
We have now obtained our partial reconstruction result~:
\begin{thm}
Let $B$ be a selfinjective algebra with no simple projective
module. Let $M$ be an $(A,B)$-bimodule inducing a stable equivalence
and having no projective direct summand. Let $\CS=\{M\otimes_B L\}$
where $L$ runs over a complete set of representatives of isomorphism
classes of simple $B$-modules.

Then, there is an equivalence
$\gr (B\mMod)\iso \CG$. If $B$ is basic, there is an isomorphism
between the graded algebra
associated with the radical filtration of $B$ and $\End_{\CG}(M)$.
\end{thm}

\subsubsection{}
The category $\CG$ can be constructed directly as in \S \ref{filtT},
using only the stable category with its triangulated structure.

\begin{prop}
\label{filtinstab}
Let $M$ be a module with an $\CS$-filtration
$0=M_r\subseteq M_{r-1}\subseteq\cdots\subseteq M_0=M$.
This is an $\CS$-radical filtration if and only if
\begin{itemize}
\item 
$\Hom_{A\mstab}(M_i/M_{i+1},S)\to \Hom_{A\mstab}(M_i,S)$ is an isomorphism
for all $S\in\CS$ and $i>0$,
\item
$\Hom_{A\mstab}(M_0/M_1,S)\to \Hom_{A\mstab}(M_0,S)$ is surjective 
for all $S\in\CS$, and
\item
$M_i$ has no projective remainder for $i>0$.
\end{itemize}
Assume the filtration is an $\CS$-radical filtration. Then $M$ has no
projective remainder if and only if
$\Hom_{A\mstab}(M_0/M_1,S)\to \Hom_{A\mstab}(M_0,S)$ is an isomorphism.
\end{prop}

\begin{proof}
Let  $M$ be a module with an $\CS$-radical filtration
$0=M_r\subseteq M_{r-1}\subseteq\cdots\subseteq M_0=M$.
The canonical map 
$\Hom_{A\mstab}(M_i/M_{i+1},S)\to \Hom_{A\mstab}(M_i,S)$ is surjective
for all $S\in\CS$, by Lemma \ref{minimalsurj}.
Note that $M_i$ has no projective remainder for $i>0$, by
Lemma \ref{projinfilt}.
It follows that the canonical map
$\Hom_{A\mstab}(M_i/M_{i+1},S)\to \Hom_{A\mstab}(M_i,S)$ is an isomorphism
for all $S\in\CS$ (Lemma \ref{filtifinj}).

Let us now prove the other implication.
Since $M_i$ has no projective remainder for $i>0$, it follows from
Lemma \ref{minimalsurj} that
$0=M_r\subseteq M_{r-1}\subseteq\cdots\subseteq M_1$ is an $\CS$-radical
filtration of $M_1$.

\smallskip
Assume the filtration is an $\CS$-radical filtration. If $M$ has
no projective remainder, then 
$\Hom_{A\mstab}(M_0/M_1,S)\to \Hom_{A\mstab}(M_0,S)$ is injective
by Lemma \ref{filtifinj}.

Assume now that $\Hom_{A\mstab}(M/M_1,S)\to \Hom_{A\mstab}(M,S)$ is
bijective. Assume $M=M'\oplus P$ with $M'$ filtrable and $P$ projective.
We have $\Hom_{A\mstab}(M/M_1,S)\iso\Hom_{A\mstab}(M,S)\iso
\Hom_{A\mstab}(M',S)$. There is a surjective map
$g:M'\to M/M_1$ with filtrable kernel such that the composition
$M\Rarr{\can} M'\Rarr{g}M/M_1$ is equal to the canonical
map $M\to M/M_1$ in the stable category, by Proposition
\ref{head}. By Lemma \ref{stableker}, 
we have $M_1\simeq \ker g\oplus P$. Since $M_1$ has no projective
remainder by the first part of the proposition, we get $P=0$, hence
$M$ has no projective remainder.
\end{proof}

Let $\CT=A\mstab$. Note that $\CS$ is determined by its image in
$\CT$ and it satisfies Hypothesis \ref{hyp4} if and only if
$\Hom_\CT(S,T)=k^{\delta_{ST}}$ for all $S,T\in\CS$ and
every object of $\CT$ is an iterated extension of objects of $\CS$.

We have a functor $\CG\to\CF$~:
it sends
a module $M$ with an $\CS$-radical filtration
$0=M_r\subseteq M_{r-1}\subseteq\cdots\subseteq M_0=M$ to
$\cdots\to 0\to M_{r-1}\to\cdots\to M_1\to M\to M/M_1$ (cf Proposition
\ref{filtinstab}).

\begin{prop}
The canonical functor $\CG\iso\CF$ is an equivalence.
\end{prop}

\begin{proof}
The functor is clearly fully faithful.

Start with $0=N_r\Rarr{f_r} N_{r-1}\to \cdots\to
N_1\Rarr{f_1} N_0\Rarr{\eps_0} M_0$. Adding a projective direct summand
to the $N_i$'s, we can lift the maps
$f_i$ to maps that are injective in the module category and such that
the successive quotients have no projective direct summands.
So we have a filtration
$0=M'_r\subseteq M'_{r-1}\subseteq\cdots\subseteq M'_1\subseteq
M'_0$ such that $M'_i/M'_{i+1}$ is stably isomorphic to a direct sum of
objects of $\CS$. Since it has no projective summand, it is actually
isomorphic to a sum of objects of $\CS$; \ie, we have an
$\CS$-filtration. Consider $i$ maximal such that $M'_i$ has a projective
remainder. Then $0=M'_r\subseteq M'_{r-1}\subseteq\cdots\subseteq M'_i$
is an $\CS$-radical filtration by Proposition \ref{filtinstab} (first part).
The second part of Proposition \ref{filtinstab} shows now that
$M'_i$ has no projective remainder, a contradiction.
So the filtration is an $\CS$-filtration.
\end{proof}

\end{document}